\documentclass[11pt]{article}
\usepackage{amsmath}
\usepackage{amssymb}
\usepackage[amsmath,thmmarks,hyperref]{ntheorem}
\usepackage{mathrsfs}
\usepackage{eucal}
\usepackage[all]{xy}
\usepackage{float}
\usepackage{stmaryrd}
\usepackage[bookmarksnumbered]{hyperref}

\usepackage{todonotes}
\usepackage{parskip}
\usepackage{enumitem}
\usepackage{eucal}
\usepackage[textwidth=14.5cm, textheight=22cm]{geometry}
\usepackage{setspace}

\newcommand{\zz}{{ \mathbb{Z} }}

\newcommand{\rr}{{ \mathbb{R} }}

\newcommand{\qq}{{ \mathbb{Q} }}

\newcommand{\ccal}{{ \mathcal{C} }}

\newcommand{\ecal}{{ \mathcal{E} }}

\newcommand{\jcal}{{ \mathcal{J} }}

\newcommand{\lcal}{{ \mathcal{L} }}
\newcommand{\mcal}{{ \mathcal{M} }}

\newcommand{\smashprod}{\wedge}
\newcommand{\co}{\colon \!}

\newcommand{\fibrep}{{ \widehat{f} }}
\newcommand{\cofrep}{{ \widehat{c} }}

\DeclareMathOperator{\id}{Id}

\DeclareMathOperator{\h}{H}
\DeclareMathOperator{\B}{B}
\DeclareMathOperator{\ho}{Ho}

\DeclareMathOperator{\leftmod}{--mod}

\DeclareMathOperator{\ext}{Ext}

\DeclareMathOperator{\homSSS}{Hom}
\renewcommand{\hom}{\homSSS}
\DeclareMathOperator{\LimSSS}{lim }
\renewcommand{\lim}{\LimSSS}
\DeclareMathOperator{\holim}{holim }
\DeclareMathOperator{\hocolim}{hocolim }

\DeclareMathOperator{\ch}{Ch}

\newtheorem{theorem}{Theorem}[section]
\newtheorem{proposition}[theorem]{Proposition}
\newtheorem{corollary}[theorem]{Corollary}
\newtheorem{lemma}[theorem]{Lemma}

\newtheorem{definition}[theorem]{Definition}
\newtheorem{ex}[theorem]{Example}

\newtheorem{rmk}[theorem]{Remark}

\theoremsymbol{\ensuremath{\rule{0.6em}{0.6em}}}
\theoremheaderfont{\bfseries}
\theorembodyfont{\normalfont}
\theoremseparator{}
\newtheorem*{proof}{Proof}

\DeclareMathOperator{\poly}{--poly--}
\DeclareMathOperator{\homog}{--homog--}
\DeclareMathOperator{\tors}{tors--}
\DeclareMathOperator{\res}{Res}
\DeclareMathOperator{\ind}{Ind}
\DeclareMathOperator{\Sp}{Sp}

\DeclareMathOperator{\Top}{{Top}}

\DeclareMathOperator{\nat}{Nat}

\DeclareMathOperator{\skel}{sk}
\DeclareMathOperator{\inc}{inc}
\DeclareMathOperator{\mor}{mor}

\newcommand{\lra}{{ \longrightarrow }}

\title{Rational Orthogonal Calculus}
\author{David Barnes}
\date{February 5, 2017}

\begin{document}
\maketitle

\begin{abstract}
\noindent We show that one can use model categories
to construct rational orthogonal calculus.
That is, given a continuous functor from vector spaces to based spaces
one can construct a tower of approximations to this functor
depending only on the rational homology type of the input functor,
whose layers are given by rational spectra with an action of $O(n)$.
By work of Greenlees and Shipley, we see that these layers are classified
by torsion $\h^*(\B SO(n))[O(n)/SO(n)]$--modules.
\end{abstract}


\section{Introduction}
Orthogonal calculus constructs a Taylor
tower for functors from vector spaces to spaces.
The $n^{th}$ layer of this tower is determined by
a spectrum with $O(n)$ action.
Orthogonal calculus has a strong geometric flavour,
for example it was essential to the
results of \cite{ALV07} which shows how the rational homology of a manifold
determines the rational homology of its space of embeddings into a
Euclidean space.
Working rationally is also central to work of Reis and Weiss \cite{RW15}.
Thus it is natural to ask if one can
construct a rationalised version of orthogonal calculus
where the tower of a functor $F$ depends only on the 
(objectwise) rational homology type of $F$.

In this paper we apply the work of \cite{barnesoman13} to construct
suitable model categories that capture the notion of rational orthogonal calculus.
In particular, we show that the layers of the rational tower are classified by rational
spectra with an action of $O(n)$.
By the work of Greenlees and Shipley \cite{GSfree}, we see that these layers are classified
by torsion $\h^*(\B SO(n); \qq)[O(n)/SO(n)]$--modules.
Thus we have a strong technical foundation for
rational orthogonal calculus and a simpler, algebraic characterisation of the layers
which should reduce the amount of effort required in future calculations.
This paper also gives a nice demonstration of how Pontryagin classes are at the heart of orthogonal calculus, as the graded ring $\h^*(\B SO(n))$
is polynomial on the Pontryagin (and Euler) classes.

The main difficulty in the work is  setting up the model structures by careful
use of Bousfield localisations. There are some subtleties involved as we are mixing
left and right Bousfield localisations. The source of the difficulty is that
in general left localisations require left properness of the model structure, but
do not preserve right properness, whilst the reverse situation
generally holds for right localisation.
For example, the model category of $\h \qq$-local spaces
in the sense of Bousfield \cite{bous75} is
created via a left localisation
of the standard model structure on spaces and
is not right proper.

\paragraph{Organisation}
We recap the basic notions of
orthogonal calculus in Section \ref{sec:orthcalc}:
polynomial and homogeneous functors, the classification theorem and the tower,
and the derivative of a functor.
Section \ref{sec:Qspec} has the main result:
the rational version of the classification of $n$-homogeneous functors: 
Theorem \ref{thm:nohomogmodel}.

The rest of the paper is where we give the details of the proof
of the main theorem.
In Section \ref{sec:modcat} we give a brief recap of the notion of Bousfield
localisation and review the construction of the
model categories used in orthogonal calculus.

In Sections \ref{sec:ratpoly} and \ref{sec:rathomog}
we establish model structures for rational $n$-polynomial and
rational $n$-homogeneous functors.
We extend the classification results of Weiss in Section \ref{sec:stablecat}
to the rational setting and finish our description of rational $n$-homogeneous functors.

\paragraph{Acknowledgements} The author wishes to thank Greg Arone, John Greenlees,
Constanze Roitzheim and Michael Weiss for numerous helpful conversations.
The author gratefully acknowledges support from the Engineering and Physical
Sciences Research Council under grant EP/M009114/1.

\section{Orthogonal calculus}\label{sec:orthcalc}
We give a brief overview of orthogonal calculus, introducing the relevant categories and definitions. The primary reference is \cite{weiss95}. 

\subsection{Continuous functors}
Let $\lcal$ be the category of finite dimensional real inner product spaces and isometries. To ensure this category is skeletally small, we assume that these vector spaces are all subspaces of some universe, $\rr^\infty$. Note that for $U \subseteq V$, $\lcal(U,V) = O(V)/O(V-U)$, where $V-U$ denotes the orthogonal complement of $U$ in $V$.
We define a new category $\jcal_0$: it has the same objects as $\lcal$
but its morphism spaces are given by  
$\jcal_0(U,V) = \lcal(U,V)_+$.

Orthogonal calculus studies \textbf{continuous functors} from $\lcal$ to
(based or unbased) topological spaces. We restrict ourselves to the based version
and consider $\Top$-enriched functors $F \co \jcal_0 \to \Top$
where $\Top$ denotes the category of based topological spaces.
Thus $F$ is a functor with the property that the induced map
of spaces
\[
F_{U,V} \co \jcal_0(U,V) \longrightarrow \Top(F(U),F(V))
\]
is continuous (and is associative, unital and 
compatible with composition of maps of vector spaces).
We denote the category of such functors $\jcal_0 \Top$.
Whenever we talk of functors from vector spaces to spaces
we mean an object of $\jcal_0 \Top$.

\subsection{Polynomial and homogeneous functors}

\begin{definition}
An object $F \in \jcal_0 \Top$ is said to be \textbf{$n$-polynomial}
if for each $V \in \jcal_0$ the inclusion map
\[
F(V) \longrightarrow \underset{0 \neq U \subseteq \rr^{n+1}}{\holim} F(U \oplus V)
\]
is a weak homotopy equivalence.
\end{definition}
This definition captures the idea of the value of $F$ at $V$ being recoverable from
the value of $F$ at vector spaces of higher dimension (and the maps between these values).

By \cite[Proposition 4.2]{weiss95} we can give an 
equivalent definition of $n$-polynomial using a
sequence of vector bundles $\gamma_n$ over $\lcal$. Let
\[
\gamma_n(U,V) =
\{
\ (f,x) \ \mid \ f \in \lcal(U,V), \ x \in \rr^n \otimes (V- f(U)) \
\}.
\]
and define $S\gamma_n(U,V)$ be the unit sphere in $\gamma_n(U,V)$.

\begin{lemma}
Define $\tau_n \co \jcal_0 \Top \to \jcal_0 \Top$ as
\[
(\tau_n F)(V) = \nat(S\gamma_{n+1}(V,-)_+, F).
\]
A functor $F$ is $n$-polynomial if and only if
$\eta_F \co F \to \tau_n F$ is an objectwise weak homotopy equivalence,
where $\eta$ is induced by the projection $S\gamma_n(V,W)_+ \to \jcal_0(V,W)$.
\end{lemma}

Using $\tau_n$ we can construct the $n$-polynomial approximation to an element of $\jcal_0 \Top$.
\begin{definition}
For $F \in \jcal_0 \Top$ let $T_n F$ be the functor
\[
T_n F = \hocolim (
F \overset{\eta_F}{\longrightarrow}
\tau_n^2 F \overset{\eta_{\tau_n F}}{\longrightarrow}
\tau_n^3 F \longrightarrow \dots
).
\]
The map $F \to T_n F$ induced by $\eta_F$ is the
$n$-polynomial approximation to $F$.
\end{definition}
This map has the desired property that if $f \co F \to G$ is a map in $\jcal_0 \Top$
with $G$ an $n$-polynomial functor, then $f$ factors over $F \to T_n F$
(up to homotopy) in a unique way (up to homotopy).

Since this is not clear from the definitions,
we give \cite[Proposition 5.1]{weiss95}.
\begin{proposition}
If $F$ is $(n-1)$-polynomial then $F$ is $n$-polynomial.
\end{proposition}
Hence we have a map (unique up to homotopy) from
$T_n E \to T_{n-1} E$ for any $E \in \jcal_0 \Top$.

\begin{definition}
A functor $F \in \jcal_o \Top$ is said to be \textbf{$n$-homogeneous} if
$F$ is $n$-polynomial and $T_{n-1} F$ is objectwise weakly equivalent to a point. 
For $E \in \jcal_0 \Top$ we let $D_n E$ denote the homotopy fibre of the map $T_n E \to T_{n-1} E$. We call $D_n E$ the $n$-homogeneous approximation to $E$.
\end{definition}
The functor $D_n E$ is $n$-homogeneous as $T_n$ and $T_{n-1}$ commute
with homotopy fibres (and finite homotopy limits in general).
In particular $T_{n-1} T_n E = T_{n-1} E$.

\subsection{The tower and the classification}

Now we are in a place to be more definite about orthogonal calculus.
Given a functor $F \in \jcal_0 \Top$, orthogonal calculus constructs the
$n$-polynomial approximations $T_n F$ and the $n$-homogenous approximations
$D_n F$. These can be arranged in a tower of fibrations (analogous to the Postnikov tower).

For each $n \geqslant 0$ there is a
fibration sequence $D_n F \to T_n F \to T_{n-1} F$, which can be arranged as below.
\[
\xymatrix@C+1cm@R-0.5cm{
& \vdots \ar[d] \\
& T_3 F \ar[d] & D_3 F \ar[l] \\
& T_2 F \ar[d] & D_2 F \ar[l] \\
& T_1 F \ar[d] & D_1 F \ar[l] \\
F \ar[r] \ar[ur] \ar[uur] \ar[uuur]
& T_0 F. \\
}
\]
Of course, we need to know much more for this tower to be useful.
For this we have Weiss's classification of the $n$-homogeneous functors
\cite[Theorem 7.3]{weiss95}.

\begin{theorem}
The full subcategory of $n$-homogeneous functors in the
homotopy category of $\jcal_0 \Top$ is equivalent to the
homotopy category of spectra with an action of $O(n)$.
Moreover, given a spectrum $\Psi$ with an action of $O(n)$
the following formula defines an $n$-homogeneous functor in
$\jcal_0 \Top$.
\[
V \mapsto \Omega^\infty ((\Psi \smashprod S^{\rr^n \otimes V})_{hO(n)})
\]
\end{theorem}
In the above, $S^{\rr^n \otimes V}$ is the one-point compactification of
$\rr^n \otimes V$. This has $O(n)$-action induced from the regular representation
of $O(n)$ on $\rr^n$. The smash product is equipped with the diagonal action of
$O(n)$, and $h O(n)$ denotes homotopy orbits.

\section{The rational classification theorem}\label{sec:Qspec}

\subsection{Motivation}

Calculations in the orthogonal calculus can often be difficult.
By considering only the rational 
information of a functor, these calculations can often be simplified, 
such as in \cite{RW15} or \cite{ALV07}.
One implementation of this idea is to alter the constructions so 
that if $f \co F \to G$ in $\jcal_0 \Top$ 
induces a levelwise isomorphism on rational homology, then
the tower of $F$ and $G$ should agree (up to weak homotopy equivalence).

It follows that we need to construct a rational $n$-polynomial
replacement functor $T_n^\qq$. This should have the property that given
any functor $F \in \jcal_0 \Top$, $T_n^\qq F$ is the closest functor
that is both rational and $n$-polynomial. We construct such a functor in
Section \ref{sec:ratpoly}. Hence we have a rational $n$-homogeneous 
functor $D_n^\qq F$ (the homotopy fibre of $T_n^\qq F \to T_{n-1}^\qq F$.
We study rational $n$-homogeneous functors in detail in Section \ref{sec:rathomog}.

A strong piece of evidence that this implementation is the correct one is
that our rational $n$-homogeneous functors are classified in terms
of rational spectra with an action of $O(n)$. 
To that end, we introduce the algebraic model for rational spectra with an action of $O(n)$ 
from \cite{GSfree} so that we can state 
the rationalised version of Weiss's classification theorem.

\subsection{The algebraic model}

The algebraic category is based on the group cohomology of $SO(n)$.
The following calculation is well-known.
The elements $p_i$ have degree $4i$ and represent the Pontryagin classes,
the element $e$ has degree $k$ and represents the Euler class.
\[
\begin{array}{rcl}
\h^*(\B SO(2k+1); \qq) &=&
\qq [p_1, \dots ,p_k] \\
\h^*(\B SO(2k); \qq) &=&
\qq [p_1, \dots ,p_{k-1}, e]
\end{array}
\]
From now on we shall omit the $\qq$ from our notation for cohomology.
We will also use the notation $W=W_{O(2)} SO(2) = O(2)/SO(2)$.

From this data we can construct a rational differential graded algebra $\h^*(\B SO(n))[W]$.
It has basis given by symbols $w^i \cdot r$,
where $r \in \h^*(\B SO(n))$ and $i$ is either 0 or 1.
The multiplication is given by the formulae below,
where $r,s \in  \h^*(\B SO(n))$ and $w(r)$ is the image of $r$ under
the ring isomorphism $w$. The unit is $w^0 \cdot 1$ and all differentials are zero.
\[
\begin{array}{rcl}
\begin{array}{rcl}
(w^0 r) (w^0 s) & = & w^0 \cdot rs \\
(w^0 r) (w^1 s) & = & w^1 \cdot w(r) s \\
\end{array}
& &
\begin{array}{rcl}
(w^1 r) (w^0 s) & = & w^1 \cdot rs \\
(w^1 r) (w^1 s) & = & w^0 \cdot w(r)s \\
\end{array}
\end{array}
\]

\begin{definition}
Let $\ch(\qq[W])$ denote the category of rational
chain complexes with an action of $W$.
The morphisms are those maps compatible with the $W$-action.

Let $\h^*(\B SO(n))[W] \leftmod$ denote the category of
$\h^*(\B SO(n))[W]$-modules.
Every such module determines an object of $\ch(\qq[W])$
by forgetting the action
of the characteristic classes.
\end{definition}

We can also describe $\h^*(\B SO(n))[W] \leftmod$ as the category of
$\h^*(\B SO(n))$-modules in $\ch(\qq[W])$, where the $W$-action on $\h^*(\B SO(n))$
is induced by the conjugation action of $W=O(n)/SO(n)$ on $SO(n)$.

We need to add one final condition to get the correct (model) category, 
see \cite{GSfree} for details.
\begin{definition}
An object $M$ of $\h^*(\B SO(n))[W] \leftmod$ is said to be \textbf{torsion}
if for each polynomial generator $x$ of $\h^*(\B SO(n))$ and each $m \in M$,
there is an $n$ such that $x^n m=0$.
The full subcategory of $\h^*(\B SO(n))[W] \leftmod$ consisting of
torsion modules is denoted $\tors \h^*(\B SO(n))[W] \leftmod$.

This category has a model structure where the weak equivalences
are the homology isomorphisms of underlying chain complexes
and the cofibrations are the monomorphisms.
\end{definition}

\subsection{The main theorem}

\begin{theorem}\label{thm:nohomogmodel}
A rational $n$-homogeneous functor from vector spaces
to based topological spaces is uniquely determined (up to homotopy)
by a torsion $\h^*(\B SO(n))[W]$--module.
That is, the following two model categories are Quillen equivalent.
\[
n \homog \jcal_0 \Top_\qq
\simeq
\tors \h^*(\B SO(n))[W] \leftmod
\]
\end{theorem}
\begin{proof}
The model category $n \homog \jcal_0 \Top_\qq$ is defined in Theorem \ref{thm:rathomog}.
By Propositions \ref{prop:QEstable1} and \ref{prop:indrational}
the model category of rational $n$-homogeneous functors is Quillen equivalent to
the model category of rational spectra with an $O(n)$-action.
By \cite[Theorem 1.1]{GSfree} the model categories of rational spectra with an $O(n)$-action
and torsion $\h^*(\B SO(n))[W]$--modules are Quillen equivalent.
\end{proof}

This is a substantial simplification of the classification in terms
of spectra with an $O(n)$-action.
For example, the first derivative of an object of $\jcal_0 \Top$ is uniquely classified
(up to homotopy) by a chain complex $N$, with an action of $W$.
Furthermore, the second derivative of an object of $\jcal_0 \Top$ is uniquely classified
(up to homotopy) by an chain complex $M$, which has an action of $\qq[e]$
($e$ has degree 2) and an action of $W$, where
$w \ast (e \cdot m) = - e \cdot (w \ast m)$.

The Quillen equivalence of Greenlees and Shipley is a composite of
a number of adjunctions, so we leave the details to the reference.
However, if $E$ is an $O(n)$-spectrum and
$M \in \tors \h^*(\B SO(n))[W] \leftmod$ is the corresponding object
(under the series of Quillen equivalences), then the homology of $M$
is determined by the relation
\[
\h_*(M) \cong
\pi^{SO(n)}_*(ESO(n)_+\smashprod E)=
\h^{SO(n)}_*(S^{A_{n}}\smashprod E).
\]
In the above, $\h^{SO(n)}_*(X)$ denotes homology
of the Borel construction applied to the spectrum $X$
(that is, the homology of the $SO(n)$-homotopy orbits of $X$)
and $A_n$ is the adjoint
representation of $SO(n)$ at the identity.
Thus $A_n$ is the $n(n-1)/2$-dimensional vector space of
skew symmetric matrices with $SO(n)$ acting by conjugation.
The cap product induces the action of $\h^*(\B SO(n))$.
The spectrum $E$ had an $O(n)$-action and we have taken
$SO(n)$-orbits (or fixed points), hence $W=O(n)/SO(n)$
acts on the homology and homotopy groups above.

The Quillen equivalence is also equipped with an Adams
spectral sequence \cite[Theorem 9.1]{GSfree} allowing one to perform calculations easily.
In the following, $r$ is the rank of the maximal torus of $O(n)$,
so this is $n/2$ for even $n$ and $(n-1)/2$ when $n$ is odd.

\begin{theorem}
Let $X$ and $Y$ be free rational $O(n)$-spectra.
There is a natural Adams spectral sequence
\[
\ext^{*,*}_{\h^*(\B SO(n))[W]}
(\pi_*^{SO(n)}(X), \pi_*^{SO(n)}(Y))
\Longrightarrow
[X,Y]_*^{O(n)}.
\]
It is a finite spectral sequence concentrated in
rows 0 to $r$ and is strongly convergent for all $X$ and $Y$.
\end{theorem}


We give some examples to illustrate Theorem \ref{thm:nohomogmodel}.
Many more calculations in rational orthogonal calculus
can be found in Reis and Weiss \cite{RW15}.

\begin{ex}
Consider the functor which sends a vector space $V$ to
$S^{nV}=S^{\rr^n \otimes V}$. This functor can also be described as $\jcal_n(0,V)$, 
see Section \ref{sec:stablecat}.
We want to find the rational $n^{th}$-derivative of this functor.
To do so, we use the Quillen equivalences of 
Propositions \ref{prop:QEstable1} and \ref{prop:indrational}.
Letting $\mathbb{L}$ indicate derived functors, we have the following diagram
of objects.
\[\xymatrix@C+1cm{
O(n)_+ \smashprod \jcal_1(0,-)
&
\ar@{|->}[l]_-{\mathbb{L} (-)\smashprod_{\jcal_n}\jcal_1}
O(n)_+ \smashprod \jcal_n(0,-)
\ar@{|->}[r]^-{\mathbb{L} \res_0^n/O(n)} &
\jcal_n(0,-)
}\]
That is, the functor $\jcal_n(0,-)$ is the image of
the $O(n)$-spectrum $O(n)_+ \smashprod \jcal_1(0,-)$ in
the homotopy category of $n$-homogeneous functors.
The $O(n)$-spectrum $O(n)_+ \smashprod \jcal_1(0,-)$ is more commonly known as
$O(n)_+ \smashprod \mathbb{S}$, where $\mathbb{S}$ denotes the
(non-equivariant) sphere spectrum

Thus the $n^{th}$-derivative of $V \mapsto S^{nV}$
is $O(n)_+ \smashprod \mathbb{S}$.
If we then work rationally, the $n^{th}$-derivative  is
$O(n)_+ \smashprod \h \qq$. Furthermore, we can identify the
algebraic model for this spectrum in
$\tors \h^*(\B SO(n))[W] \leftmod$.
Taking homotopy groups we obtain
\[
\pi_*^{SO(n)} (O(n)_+ \smashprod \mathbb{S}) =
\h_*^{SO(n)}(S^{A_n} \smashprod SO(n)_+)[W]
\cong \Sigma^{n(n-1)/2} \qq[W]
\]
with the Pontryagin (and Euler) classes acting as zero for degree reasons.
If $X$ is a chain complex having this homology, then
$X \simeq \Sigma^{n(n-1)/2} \qq[W]$ by a simple formality argument.
Hence the rational $n^{th}$-derivative of $V \mapsto S^{nV}$
is given by $\Sigma^{n(n-1)/2} \qq[W]$.
\end{ex}

\begin{ex}
The previous example is the case $U=0$ of the representable functor 
$V \mapsto \jcal_n(U,V)$ as defined in Section \ref{sec:stablecat}.
This functor from vector spaces to spaces is like a shift desuspension of 
$V \mapsto \jcal_n(0,V)=S^{nV}$. Indeed, the corresponding object 
to $\jcal_n(U,-)$ in the 
homotopy category of $n$-homogeneous functors is the functor
$V \mapsto O(n)_+ \smashprod \jcal_n(U,V)$, which \emph{is} 
the shift desuspension of $\jcal_n(0,-)$. As the $n$-homogeneous
model structure is stable, we want to find the algebraic model for the 
desuspension of $V \mapsto \jcal_n(0,V) =S^{nV}$ by the sphere $S^{nU}$, 
where $O(n)$ acts on the $\rr^n$ in $S^{nU} = S^{\rr^N \otimes U}$. 

The $O(n)$-spectrum $S^{n\rr}$ corresponds to the object $\widetilde{Q}[n]$ in
the algebraic model: the sign representation of $O(n)$ in degree $n$
with zero differential. Similarly $S^{nU}$ for $U$ a vector space of dimension $u$
corresponds to $\widetilde{\qq}^{\otimes u}[nu]$. So for even $u$,
this is just $\qq$ in degree $nu$ and for odd $u$ it is $\widetilde{Q}$
in degree $nu$.

It follows that the algebraic model for $\jcal_n(U,-)$ is the 
desuspension of  $\Sigma^{n(n-1)/2} \qq[W]$ given by 
tensoring (over $\qq$) with  $\widetilde{\qq}^{\otimes u}[-nu]$.
As $\qq \otimes \qq[W] = \qq[W] = \qq[W] \otimes \widetilde{\qq}$, 
we see that $\jcal_n(U,-)$ corresponds to $\Sigma^{n(n-1)/2 -nu} \qq[W]$.
\end{ex}

\begin{ex}
A related example is the functor which sends
$V$ to $S^{nV}/hO(n)$, where $O(n)$ acts on
$\rr^n \otimes V$ by the standard action.

The $n^{th}$-derivative of this in $O(n) \ltimes (\jcal_n \Top)$
is $\jcal_n(0,-)$, which in turn corresponds to the sphere spectrum
(with trivial action) in $O(n)\Sp$.
So rationally the derivative is $\h \qq$.
The algebraic model for this spectrum has homology given by
$\h_*^{SO(n)}(S^{A_n})$, where $A_n$ is the adjoint
representation of $SO(n)$ at the identity.
If we ignore the action of $W$, this is
essentially a suspension of $\h_*(\B SO(n))$
with the Pontryagin classes acting by the cap product.
Hence the rational $n^{th}$-derivative of $V \mapsto S^{nV}/hO(n)$
is (a twisted suspension) of $\h_*(\B SO(n))$.
\end{ex}

\begin{ex}
The functor $V \mapsto BO(V)$ is often considered as 
the orthogonal calculus equivalent of the identity functor. 
It has first derivative the 
sphere spectrum and second derivative the desuspension of the 
by sphere spectrum (each with trivial action) by \cite[Example 2.7]{weiss95}.
The remaining are rationally trivial by \cite[Theorem 4]{Arone02}.
It follows that the algebraic model for derivatives of $BO$ 
are $\qq$ in degree 0,  $\qq$ in degree -1 and $0$ for higher derivatives.
\end{ex}

\section{Model categories for orthogonal calculus}\label{sec:modcat}

To prove the rational classification theorem we use the language of model
categories and Quillen functors. 
In this section we recall the construction of the $n$-polynomial and $n$-homogeneous
model structures from \cite{barnesoman13}. We start with some basic model category notions.

\subsection{Cellular, topological and proper model categories}

Just as one has a simplicial model category,
there is the notion of a topological model category.
See \cite[Definition 5.12]{mmss01} for more details.

\begin{definition} \label{def:topological}
Let $\mcal$ be a model category that is enriched, tensored and cotensored in
(based) topological spaces, with $\hom_\mcal(-,-)$ denoting the enrichment.
For maps $i: A \to X$ and $p: E \to B$ in $\mathcal{M}$,
there is a map of spaces
\[
\hom_\mathcal{M} (i^\ast, p_\ast): \hom_\mathcal{M}(X, E) \to \hom_\mathcal{M} (A, E) \times_{\hom_\mathcal{M} (A,B)} \hom_\mathcal{M} (X,B)
\]
induced by $\hom_\mathcal{M}(i, id)$ and
$\hom_\mathcal{M}(i, p)$.
We say that $\mathscr{M}$ is a \textbf{topological model category}
if $\mathcal{M} (i^\ast, p_\ast)$ is a Serre fibration of spaces
whenever $i$ is a cofibration and $p$ is a fibration and further that
$\mathcal{M} (i^\ast, p_\ast)$
is a weak equivalence if (in addition) either $i$ or $p$ is a weak equivalence.
\end{definition}

\begin{definition}
Let $\mathcal{M}$ be a model category and let the following be a commutative square in $\mathcal{M}$:
\[
\xymatrix{
A \ar[r]^f \ar[d]_i & B \ar[d]^j\\
C \ar[r]_g & D\\
}.
\]
$\mathcal{M}$ is called \textbf{left proper} if, whenever $f$ is a weak equivalence, $i$ a cofibration and the square is a pushout, then $g$ is also a weak equivalence.
$\mathcal{M}$ is called \textbf{right proper} if, whenever $g$ is a weak equivalence, $j$ a fibration, and the square is a pullback, then $f$ is also a weak equivalence.
$\mathcal{M}$ is called \textbf{proper} if it is both left and right proper.
\end{definition}

The definition of a cellular model category \cite[Definition 12.1.1]{hir03} 
is a complicated extension of the notion of cofibrantly generated model category, so we
define the easier notion \cite[Definition 13.2.1]{hir03}
and leave the extra details to the reference.

\begin{definition}
A \textbf{cofibrantly generated model category} is a model category $\mcal$ with sets of maps $I$ and $J$ such that $I$ and $J$ support the small object argument (see \cite[Definitions 15.1.1 and 15.1.7.]{mp12}) and
\begin{enumerate}
\item a map is a trivial fibration if and only if it has the right lifting property with respect to every element of $I$, and
\item a map is a fibration if and only if it has the right lifting property with respect to every element of $J$.
\end{enumerate}
A \textbf{cellular model category} is a cofibrantly generated model category where
further technical restrictions (about factoring through subcomplexes and cofibrations being effective monomorphisms) are placed on the sets of generating cofibrations and acyclic cofibrations.
\end{definition}

Relevant examples of cellular model categories include:
simplicial sets, topological spaces, sequential
spectra, symmetric spectra and orthogonal spectra (with a group acting).
These are all proper and (with the exception of simplicial sets) topological.

\subsection{Bousfield localisations}

With the exception of Section \ref{sec:Qspec} we will be using topological model
categories. These categories have the advantage that we can use the enrichment in
topological spaces to define the weak equivalences of left or right localisations.
This avoids the more complicated terminology of
homotopy function complexes, see \cite[Section 17]{hir03}.
The essential point is that for cofibrant $X$ and fibrant $Y$ in a
topological model category $\mcal$ we have
\[
\pi_n \hom_\mcal(X,Y) = [S^n \otimes  X, Y]^{\mcal}.
\]
Where $\hom_\mcal(-,-)$ denotes the enrichment, $\otimes$ is the tensoring, and $[-,-]^{\mcal}$ is maps in the homotopy category of $\mcal$.
For general $X$ and $Y$ in $\mcal$ we define
\[
\rr \hom_\mcal(X,Y) =\hom_\mcal(\cofrep X,\fibrep Y)
\]
where $\cofrep$ denotes cofibrant replacement
and $\fibrep$ denotes fibrant replacement in $\mcal$.

We now summarise Hirschhorn's results on left and right Bousfield localisations,
see \cite[Sections 4 and 5]{hir03}.
The techniques of Bousfield localisation of model categories allow one to construct a new model structure from a given model category with a larger class of weak equivalences. This was used in \cite{barnesoman13} to construct the
$n$-polynomial and $n$-homogeneous model structures on $\jcal_0 \Top$
from the objectwise model structure. We begin with the necessary terminology.

\begin{definition}
  Let $S$ be a set of maps in a topological model category $\mcal$.
  An object $Z$ is said to be \textbf{$S$-local} if
  for any map $s \co A \to B$ in $S$, there is a weak homotopy
  equivalence of spaces
  \[
  s^* \co \rr \hom_\mcal(B,Z) \lra \rr \hom_\mcal(A,Z).
  \]
  A map $f \co X \to Y$ is said to be an \textbf{$S$-equivalence} if for any
  $S$-fibrant object $Z$ there is  a weak homotopy
  equivalence of spaces
  \[
  f^* \co \rr \hom_\mcal(Y,Z) \lra \rr \hom_\mcal(X,Z).
  \]
  We say that $Z$ is \textbf{$S$-fibrant} if it is $S$-local and fibrant in $\mcal$.
\end{definition}

We have an almost dual set of definitions for right localisations.
The $K$-equivalences we define below are sometimes called
$K$-coequivalences or $K$-colocal equivalences.

\begin{definition}
  Let $K$ be a set of cofibrant objects in a topological model category $\mcal$.
  A map $f \co X \to Y$ is said to be a \textbf{$K$-equivalence} if for any
  $k \in K$ there is  a weak homotopy
  equivalence of spaces
  \[
  f_* \co \rr \hom_\mcal(k, X) \lra \rr \hom_\mcal(k, Y).
  \]
  An object $C$ is said to be \textbf{$K$-colocal} if for any
  $K$-equivalence $f \co X \to Y$ in $S$, there is a weak homotopy
  equivalence of spaces
  \[
  f_* \co \rr \hom_\mcal(C, X) \lra \rr \hom_\mcal(C, Y).
  \]
  We say that $C$ is \textbf{$K$-cofibrant} if it is
  $K$-colocal and cofibrant in $\mcal$.
\end{definition}

Notice that every weak equivalence of $\mcal$ is both a $K$-equivalence and an $S$-equivalence for any set of objects $K$ and any set of maps $S$.

\begin{theorem}
  Let $\mcal$ be a cellular and left proper topological model category
  and $S$ be a set of cofibrations. There
  is a cellular and left proper topological model category
  $L_S \mcal$, with the same cofibrations as $\mcal$, whose weak equivalences
  are the $S$-equivalences. The class of fibrant objects is the
  class of $S$-fibrant objects.
\end{theorem}

\begin{theorem}
  Let $\mcal$ be a cellular and right proper topological model category
  and $K$ be a set of cofibrant objects. There
  is a right proper topological model category
  $R_K \mcal$, with the same fibrations as $\mcal$, whose weak equivalences
  are the $K$-equivalences. The class of cofibrant objects is the class of
  $K$-cofibrant objects.
\end{theorem}

It is immediate that the identity functor is a left Quillen functor from
$\mcal \to L_S \mcal$ and a left Quillen functor from
$R_K \mcal \to \mcal$.

\subsection{Application to orthogonal calculus}

The model structures used in \cite{barnesoman13} are defined using Bousfield localisations of (more standard) model structures. We give the relevant results
and indicate where we have Quillen equivalences.

We begin with the objectwise or projective model structure on
continuous functors from $\jcal_0$ to $\Top$. We let $\skel \lcal$ denote a skeleton of 
the category of finite dimensional vector spaces and isometries (it is also a skeleton
of $\jcal_0$). 
\begin{proposition}\label{prop:objecwise}
There is an \textbf{objectwise model structure} on $\jcal_0 \Top$
where a fibration is an objectwise Serre fibration of based spaces
and a weak equivalence is an objectwise weak homotopy equivalence.
This model category is cellular and proper with
generating cofibrations and acyclic cofibrations given by
\[
\begin{array}{rcl}
I_{l} & = & \{
\jcal_0(V,-) \smashprod i \ | \
V \in \skel \lcal, \
i \in I_{\Top} \} \\
J_{l} & = & \{
\jcal_0(V,-) \smashprod j \ | \
V \in \skel \lcal, \
j \in J_{\Top} \}.
\end{array}
\]
\end{proposition}
\begin{proof}
See \cite[Theorem 6.5]{mmss01}.
\end{proof}

We now alter the objectwise model structure on $\jcal_0 \Top$ to obtain
a model category whose homotopy category is the category of
$n$-polynomial objects up to homotopy. For details
see \cite[Section 6]{barnesoman13}. In particular that reference shows that 
the model structure is right proper, despite it being a left Bousfield localisation.

\begin{proposition}
The \textbf{$n$-polynomial model structure} on $\jcal_0 \Top$
is the left Bousfield localisation of the objectwise model structure
at the set of maps
\[
S_n =
\{
s_n^V \co S\gamma_{n+1} (V, -)_+ \longrightarrow \jcal_0(V,-) \ | \ V \in \skel \lcal
\}.
\]
The fibrant objects are the $n$-polynomial functors,
the cofibrations are as for the objectwise model structure
and the weak equivalences are those maps $f$ such that $T_n f$
is an objectwise weak equivalence. The natural transformation $\id \to T_n$
gives a fibrant replacement functor.
We denote this model structure by $n \poly \jcal_0 \Top$. It is
proper and cellular.
\end{proposition}

We construct the $n$-homogeneous model structure from the $n$-polynomial model
structure using a right Bousfield localisation.

\begin{definition}
We define $n \homog \jcal_0 \Top$, the
\textbf{$n$-homogeneous model structure} on $\jcal_0 \Top$
to be the right Bousfield localisation of $n \poly \jcal_0 \Top$
at the set of objects
\[
K_n = \{ \jcal_n(V,-) \ | \ V \in \skel \lcal
\}.
\]
\end{definition}

The following result summarises \cite[Proposition 6.9]{barnesoman13}.
The weak equivalences of the $n$-homogeneous model structure 
may also be described as those maps $f$ such that $D_n F$ is an objectwise weak equivalence.

\begin{proposition}\label{prop:nhomogmodel}
The cofibrant-fibrant objects of the $n$-homogeneous model structure
$n \homog \jcal_0 \Top$
are the $n$-homogeneous objects of
$\jcal_0 \Top$ which are cofibrant in the objectwise model structure.
The fibrations are the same as in the $n$-polynomial model structure.
The weak equivalences are those maps $f$ such that $\res_0^n \ind_0^n T_n f$
is an objectwise weak homotopy equivalence.
This model structure is right proper.
\end{proposition}

To summarise, we have a diagram of model structures and Quillen adjunctions
\[
\xymatrix@C+1cm{
n \homog \jcal_0 \Top \ar@<+1ex>[r]^-{\id} &
\ar@<+1ex>[l]^-{\id}
n \poly \jcal_0 \Top  \ar@<+1ex>[r]^-{\id} &
\ar@<+1ex>[l]^-{\id}
(n-1) \poly \jcal_0 \Top
}
\]
whose homotopy categories and derived functors are
\[
\xymatrix@C+1cm{
\txt{$n$-homogeneous\\functors}
\ar@<+1ex>[r]^-{\inc} &
\ar@<+1ex>[l]^-{D_n}
\txt{$n$-poynomial\\functors}  \ar@<+1ex>[r]^-{T_{n-1}} &
\ar@<+1ex>[l]^-{\inc}
\txt{$(n-1)$-poynomial\\functors.}
}
\]

\section{Rational polynomial functors}\label{sec:ratpoly}
In this section we construct a category of rational $n$-polynomial
functors, that is, a model category whose fibrant objects are both
(objectwise) rational and $n$-polynomial. We begin by recapping
the construction of a model category of rational spaces.

\subsection{Rational spaces}

\begin{proposition}
There is a \textbf{rational model structure}
on based topological spaces where the weak equivalences are those
maps which induce isomorphisms on (reduced) rational homology.
The fibrant objects are the $\h \qq$-local spaces. The cofibrations
are as for the Serre model structure.
This model structure is cellular and left proper, with generating acyclic
cofibrations denoted by $J_{\qq \Top}$.
We denote this model structure $L_{\h \qq} \Top$.
\end{proposition}
\begin{proof}
Such a model structure exists by \cite{bous75}.
To see that it is cellular, we use
\cite[Theorem 4.1.1]{hir03} and
\cite[Example E.4]{far96}, which identifies a map $f$
such that being $f$-local is equivalent to being $\h\qq$-local.
\end{proof}

\begin{rmk}
A simply connected space that is $\h \qq$-local
is usually called a rational space.
In this paper we do not make any assumption on the connectivity of our spaces.

Working with $\h \qq$-local spaces (as opposed to rational spaces)
has two main advantages.
Firstly we can consider non-nilpotent (and non-simply connected) spaces
such as $BO(n)$ without having to use the category of spaces
over some classifying space $BK$ as in \cite{RW15}.
Secondly, the existence of well-behaved model structures allows us to phrase the
classification of rational $n$-homogeneous functors in terms of Quillen equivalences
and make use of the existing work
on model categories for orthogonal calculus.
We note Walter \cite{walter06} takes the other approach and
studies homotopy functor calculus of rational spaces in the setting of Goodwillie calculus.
\end{rmk}

We extend this model structure on spaces to $\jcal_0 \Top$
in the expected manner.

\begin{proposition}
There is an \textbf{objectwise rational model structure}
on $\jcal_0 \Top$, denoted $\jcal_0 \Top_\qq$.
The fibrations are the those maps which are objectwise
fibrations of $L_{\h \qq} \Top$.
The fibrant objects are the objectwise $\h \qq$-local spaces
and the weak equivalences are those maps which
induce objectwise weak equivalences on rational homology.
The cofibrations are as for the objectwise model structure on
$\jcal_0 \Top$.
This model structure is a left Bousfield localisation
of the objectwise model structure.
\end{proposition}
\begin{proof}
We let the generating cofibrations be the set $I_l$ from the objectwise model structure of Proposition \ref{prop:objecwise}
and the generating acyclic cofibrations be the set of morphisms
\[
\jcal_0 \smashprod J_{\qq \Top} =
\{
\jcal_0(V,-) \smashprod j \ | \ j \in J_{\qq \Top} \ V \in \skel \lcal
\}.
\]
These sets define a cellular and left proper model structure by the same argument as for
the objectwise model structure. The statements about the fibrant objects and
weak equivalences are immediate.

It is clear that this model structure is
the localisation of the objectwise model structure at the set
of morphisms $\jcal_0 \smashprod J_{\qq \Top}$.
\end{proof}

\subsection{Double localisations}

We will construct a model category of rational $n$-polynomial
functors as a double localisation of the objectwise model structure
on $\jcal_0 \Top$. One localisation is at the set $S_n$, which makes the fibrant
objects $n$-polynomial. The other is localisation at the set
$\jcal_0 \smashprod J_{\qq \Top}$, which makes
the fibrant objects $\h \qq$-local.

We give a couple of lemmas examining the structure
of double localisations. The first is essentially the language of
Bousfield lattices adapted to more general localisations.
In the following we let $\fibrep$ denote fibrant replacement in $\ccal$,
$\fibrep_S$ denote fibrant replacement in $L_S \ccal$ and
$\fibrep_T$ denote fibrant replacement in $L_T \ccal$.

\begin{lemma}\label{lem:doublelocal}
Let $\ccal$ be a left proper and cellular topological model
category and $S$ and $T$ be sets of maps in $\ccal$.
Then we have equalities of model structures
\[
L_S (L_T \ccal) = L_{S \cup T} \ccal = L_T (L_S \ccal).
\]
In each case the class of fibrant objects is the class of
fibrant objects of $\ccal$ which are both $S$-local and $T$-local.
The weak equivalences are the $S \cup T$-equivalences
and include both the $S$-equivalences and the $T$-equivalences.
\end{lemma}
\begin{proof}
The model structures exist and are left proper and cellular by
\cite[Theorem 4.1.1]{hir03}.
The fibrant objects in $L_S (L_T \ccal)$ are precisely the fibrant objects
of $L_T \ccal$ that are $S$-local in $L_T \ccal$. That is,
$C$ is fibrant in $L_S (L_T \ccal)$ if and only if it is fibrant in $\ccal$,
$T$-local and for any map $s \co A \to B$ in $S$, the induced map
\begin{equation}
s^* \co \rr \hom_{L_T \ccal}(B,C) \lra \rr \hom_{L_T \ccal}(A,C) \tag{1}
\end{equation}
is a weak homotopy equivalence of spaces.
For any $X$ and $Y$, we have
\[
\rr \hom_{L_T \ccal}(X,Y)  : = \hom_{\ccal}(\cofrep X,\fibrep_T Y)
= \rr \hom_{\ccal}(X,\fibrep_T Y).
\]
Since $C$ is $T$-local, $C \simeq \fibrep_T C$, hence
the condition (1) is exactly the statement that
$C$ is $S$-local in $\ccal$.

By symmetry, the fibrant objects of $L_S (L_T \ccal)$ are exactly the fibrant
objects of $L_{S \cup T} \ccal$. It follows that these two model categories
have the same weak equivalences (and the same cofibrations as $\ccal$),
so they are equal.

The weak equivalences are the $S \cup T$-local equivalences by definition.
Any $S$-local equivalence is a weak equivalence in $L_T(L_S \ccal)$ as
this is a Bousfield localisation. Similarly every $T$-local equivalence
is a weak equivalence in $L_T(L_S \ccal)$.
\end{proof}

It is important to note that we do not claim
that the fibrant replacement functors $\fibrep_S$ and $\fibrep_T$ commute.
In general, this will be false (consider localising spectra at $\h \qq$ and
$M \zz/p$). We will add an assumption about the interaction of the
fibrant replacement functors. This assumption will hold in the case of rational
$n$-polynomial functors.

\begin{lemma}\label{lem:doublefibrant}
Let $\ccal$ be a left proper and cellular model
category and $S$ and $T$ be sets of cofibrations in $\ccal$.
Assume that $\fibrep_T$ preserves $S$-local objects of $\ccal$.
Then fibrant replacement in $L_{S \cup T} \ccal$ can be taken to be the composite
natural transformation
\[\xymatrix@C+0.6cm{
\id \ar[r]^-{\nu_S} & \fibrep_S
\ar[r]^-{\nu_T \fibrep_S} & \fibrep_T \circ \fibrep_S.
}\]
The weak equivalences are characterised as $T$-equivalences between
$S$-localisations. Moreover, $f$ is an $S \cup T$-local equivalence
if and only if $\fibrep_S f$ is a $T$-local equivalence.
\end{lemma}
\begin{proof}
The description of fibrant replacements is routine, so we turn to the statement
about weak equivalences.
Let $f \co C \to C'$ be an $S \cup T$-local equivalence. Then we have a
commutative diagram of maps in $\ccal$ as below.
\[
\xymatrix{
C \ar[r]^{f}_{\simeq_{S \cup T}} \ar[d]_{\simeq_S} & C' \ar[d]_{\simeq_S} \\
\fibrep_S C \ar[r]^{\fibrep_S f} \ar[d]_{\simeq_T} & \fibrep_S C' \ar[d]_{\simeq_T} \\
\fibrep_T \fibrep_S C \ar[r]^{\fibrep_T \fibrep_S f}
& \fibrep_T \fibrep_S C'
}
\]
The lower horizontal map is a $S \cup T$-local equivalence between
$S \cup T$-local objects and hence is a weak equivalence in $\ccal$.
It follows that $\fibrep_S f$ is a $T$-local equivalence.
The converse is immediate.
\end{proof}

\subsection{Rational polynomial functors}

\begin{theorem}\label{thm:ratpoly}
There is a \textbf{rational $n$-polynomial model structure}
on $\jcal_0 \Top$, denoted $n \poly \jcal_0 \Top_\qq$.
The fibrant objects are the $n$-polynomial
objects which are $\h \qq$-local.
The cofibrations are as for the objectwise model structure on
$\jcal_0 \Top$.
This model structure is cellular and left proper,
fibrant replacement is given by $T_n^\qq = T_n \circ \fibrep_{\h \qq}$
and the weak equivalences are those maps which induce $T_n$-equivalences
of $\h \qq$-localisations.
\end{theorem}
\begin{proof}
Recall $J_{\qq \Top}$, the set of generating
acyclic cofibrations for the rational model structure on based spaces.
We extended these to $\jcal_0 \Top$ as the set
\[
\jcal_0 \smashprod J_{\qq \Top}
=
\{
\jcal_0(V,-) \smashprod j \ | \ j \in J_{\qq \Top} \ V \in \skel \lcal
\}.
\]
We define the rational $n$-polynomial model structure as
\[
n \poly \jcal_0 \Top_\qq =
L_{S_n} L_{\jcal_0 \smashprod J_{\qq \Top}} \jcal_0 \Top =
L_{S_n} \jcal_0 \Top_\qq.
\]
To complete the proof, we apply Lemmas \ref{lem:doublelocal} and \ref{lem:doublefibrant}. Thus we must show that
if $F$ is objectwise $\h \qq$-local then so is $T_n F$.
Let $F$ be objectwise $\h \qq$-local. We claim that for any $V \in \jcal_0$, the space
\[
\tau_n F(V) = \nat(S\gamma_n(V,-)_+, F)
\]
is $\h \qq$-local, so that $\tau_n F$ is
objectwise $\h \qq$-local.
That is, we want to show that
\[
\nat(S\gamma_n(V,-)_+, -) \co \jcal_0 \Top_\qq \longrightarrow L_{\h \qq} \Top
\]
is a right Quillen functor. This functor has a left adjoint
$S\gamma_n(V,-)_+ \smashprod (-)$, which is a special case
of the objectwise smash product
\[
(-) \smashprod (-) \co \jcal_0 \Top_\qq \times L_{\h \qq} \Top
\longrightarrow \jcal_0 \Top_\qq.
\]
It is easily checked that this smash product is a Quillen bifunctor
with respect to the $\h \qq$-local model structures on $\jcal_0 \Top$ and $\Top$.
Hence the claim follows.

The functor $T_n F$ is the homotopy colimit over $k$ of
$\tau_n^k F$ and this homotopy colimit is constructed objectwise.
Since homotopy colimits preserve $\h \qq$-local spaces, we see that $T_n F$ is
$\h \qq$-local.
\end{proof}

For more details on the interactions between
localisations and Quillen bifunctors, see \cite{GR14}.

\begin{corollary}
Let $F \in \jcal_0 \Top$. Then the homotopy fibre of
\[
T_n (\fibrep_{\h\qq} F) \longrightarrow T_{n-1} (\fibrep_{\h\qq} F)
\]
is an $n$-homogeneous functor that is objectwise $\h \qq$-local.
We denote this $D_n^\qq F$.
\end{corollary}

We want to show that such functors are classified by rational spectra
with an action of $O(n)$. We do so by proving that there is a Quillen equivalence
between a model structure for rational $n$-homogeneous functors and
a model structure for rational spectra with an action of $O(n)$.

\begin{rmk}
  A formal consequence of definitions is that for $F \in \jcal_0 \Top$, 
  $T_n^\qq F = T_n L_{\h \qq} F$ is the closest
  functor to $F$ that is both $n$-polynomial and objectwise $\h \qq$-local. 
  But it should be noted that for $F \in \jcal_0 \Top$, the functor $\fibrep_{\h\qq} T_n F$ is
  weakly equivalent to $F$ in $n \poly \jcal_0 \Top_\qq$, but it is not necessarily $n$-polynomial.
 \end{rmk}

\section{Rational homogeneous functors}\label{sec:rathomog}

In this section we will left Bousfield
localise the $n$-homogeneous model structure on $\jcal_0 \Top$
to ensure that the fibrant objects are also objectwise $L_{\h \qq}$-local.
This requires some care to ensure both that the $n$-homogeneous model structure
admits a left Bousfield localisation and that a localisation
with the correct properties exists.

In order that we can perform left Bousfield localisations on a model structure we
need to know that it is left proper and cellular.
The following lemma builds upon Proposition \ref{prop:nhomogmodel}
and shows that these properties hold for $n \homog \jcal_0 \Top$.
The key fact is that this model structure is \textbf{stable}:
the (objectwise) suspension is an equivalence on the homotopy
category. This follows from the fact that any homogeneous 
functor has a de-looping, see also \cite[Corollary 10.2]{barnesoman13}.

\begin{lemma}
The pushout of an $n$--homogeneous weak equivalence along
a cofibration (of the objectwise model structure)
is an $n$--homogeneous weak equivalence.
In particular, the $n$--homogeneous model
structure on $\jcal_0 \Top$ is left proper (and hence is proper).
Moreover this model structure is cellular.
\end{lemma}
\begin{proof}
These statements follow from the proofs of
\cite[Proposition 5.8 and Theorem 5.9]{barnesroitzheimstable}.
In each case we do not need the original category ($n \poly \jcal_0 \Top$) to be stable, only that its right localisation ($n \homog \jcal_0 \Top$) is stable.
\end{proof}

Now we need to specify a set of maps at which to localise
$n \homog \jcal_0 \Top$. A reasonable first guess would be
$\jcal_0 \smashprod J_{\qq \Top}$ as we used this
set to make the model structure of rational $n$-polynomial functors.
But this set would cause us some technical problems as the functor $\jcal_0 (U,-)$ is
(most likely) not cofibrant in $n \homog \jcal_0 \Top$. Instead the
cofibrant objects are built (in a well-defined sense)
from the objects $\jcal_n(U,-)$
see \cite[Definition 5.1.4 and Theorem 5.1.5]{hir03}.
Thus we will left Bousfield localise $n \homog \jcal_0 \Top$ at the set
\[
\jcal_n \smashprod J_{\qq \Top}
=
\{
\jcal_n(V,-) \smashprod j \ | \ j \in J_{\qq \Top} \ V \in \skel \lcal
\}.
\]
If we had chosen the set $\jcal_0 \smashprod J_{\qq \Top}$
then we would need to cofibrantly replace the codomain and
domain in the definition of the localised weak equivalences.
This would make it much harder to understand the fibrant objects.
As it is, we will not see until Proposition \ref{prop:indrational}
that this model structure has all the properties we desire.

\begin{theorem}\label{thm:rathomog}
There is a \textbf{rational $n$-homogenous model structure}
on $\jcal_0 \Top$, denoted $n \homog \jcal_0 \Top_\qq$.
This model structure is cellular and proper.
The fibrant objects are the $n$-homogeneous functors $F$
such that $\ind_0^n F$ is objectwise $\h \qq$-local.
\end{theorem}
\begin{proof}
We define the model category $n \homog \jcal_0 \Top_\qq$
to be the left Bousfield localisation of
$n \homog \jcal_0 \Top$ at the set of maps
\[
\jcal_n \smashprod J_{\qq \Top}
=
\{
\jcal_n(V,-) \smashprod j \ | \ j \in J_{\qq \Top} \ V \in \skel \lcal
\}.
\]
This model structure exists as we are in a cellular and proper model category.
Furthermore the resulting model structure is cellular and proper by
\cite[Proposition 5.8 and Theorem 5.9]{barnesroitzheimstable}.

The set $\jcal_n \smashprod J_{\qq \Top}$ consists of cofibrations in
$n \homog \jcal_0 \Top$, since $\jcal_n(U,-)$ is cofibrant and
$J_{\qq \Top}$ is a set of cofibrations in $\Top$.
By \cite[Theorem 4.11]{barnesroitzheimstable} it follows that
the generating acyclic cofibrations of $n \homog \jcal_0 \Top_\qq$
are given by taking the
generating acyclic cofibrations of $n \homog \jcal_0 \Top$ and adding
the set of horns (see \cite[Definition 4.2.1]{hir03}) on $\jcal_n \smashprod J_{\qq \Top}$.
In a topological model category, a horn is a type of pushout, so it 
suffices to use simply $\jcal_n \smashprod J_{\qq \Top}$.
Thus a fibration in $n \homog \jcal_0 \Top_\qq$ is precisely a map
which is a fibration in $n \poly \jcal_0 \Top$ with the right lifting
property with respect to $\jcal_n \smashprod J_{\qq \Top}$.
A map $E \to F$ has the right lifting property
with respect to $\jcal_n \smashprod J_{\qq \Top}$ if and only if
$\ind_0^n E(U) \to \ind_0^n F(U)$ is a rational fibration of based spaces.

So we have shown that the fibrations of
$n \homog \jcal_0 \Top_\qq$ are precisely the fibrations of
$n \homog \jcal_0 \Top$ such that $\ind_0^n f$ is an objectwise
rational fibration.
It follows that the fibrant objects are the $n$-polynomial functors $F$
such that $\ind_0^n F$ is objectwise $\h \qq$-local.
\end{proof}

\begin{lemma}
The identity functor is a left Quillen functor from
$n \homog \jcal_0 \Top_\qq$ to
$n \poly \jcal_0 \Top_\qq$.
\end{lemma}
\begin{proof}
We have localised the first category with respect to
$\jcal_n \smashprod J_{\qq \Top}$.
Recall that $\jcal_n(U,-)$ is cofibrant in $n \poly \jcal_0 \Top_\qq$
by \cite[Lemma 6.2]{barnesoman13} (in the notation of that paper
$\jcal_n(U,V) =\mor_n(U,V)$).
As mentioned in the proof of Theorem \ref{thm:ratpoly}
the objectwise smash product gives a Quillen bifunctor from
$n \poly \jcal_0 \Top_\qq \times \Top_\qq$ to
$n \poly \jcal_0 \Top_\qq$.
Hence the maps in $\jcal_n \smashprod J_{\qq \Top}$
are acyclic cofibrations in the rational
$n$-polynomial model structure
and we have a left adjoint as claimed.
\end{proof}

\begin{corollary}
  Let $F \in \jcal_0 \Top$. The cofibrant replacement of $D_n^\qq F$
  in the levelwise model structure on $\jcal_0 \Top$ is a fibrant-cofibrant
  object of $n \homog \jcal_0 \Top_\qq$.
\end{corollary}

To summarise, we have a diagram of model structures and Quillen adjunctions
\[
\xymatrix@C+1cm{
n \homog \jcal_0 \Top_\qq \ar@<+1ex>[r]^-{\id} &
\ar@<+1ex>[l]^-{\id}
n \poly \jcal_0 \Top_\qq  \ar@<+1ex>[r]^-{\id} &
\ar@<+1ex>[l]^-{\id}
(n-1) \poly \jcal_0 \Top_\qq
}
\]
whose homotopy categories and derived functors are
\[
\xymatrix@C+1cm{
\txt{rational\\$n$-homogeneous\\functors}
\ar@<+1ex>[r]^-{\inc} &
\ar@<+1ex>[l]^-{D_n^\qq}
\txt{rational\\$n$-poynomial\\functors}  \ar@<+1ex>[r]^-{T_{n-1}^\qq} &
\ar@<+1ex>[l]^-{\inc}
\txt{rational\\$(n-1)$-poynomial\\functors.}
}
\]

\section{The stable categories}\label{sec:stablecat}
Weiss' classification of $n$-homogeneous functors
in terms of spectra with an $O(n)$-action was lifted to the level of Quillen
equivalences of model categories in \cite[Section 10]{barnesoman13}.
In this section we extend this result to the rationalised case.
We start by introducing $O(n) \ltimes (\jcal_n \Top)$,
an intermediate category between 
spectra with an action of $O(n)$ and $n$-homogeneous functors.

\subsection{The intermediate categories}

Recall the vector bundles $\gamma_n$ over $\lcal$:
\[
\gamma_n(U,V) =
\{
\ (f,x) \ \mid \ f \in \lcal(U,V), \ x \in \rr^n \otimes (V- f(U)) \
\}.
\]
Define $\jcal_n(U,V)$ to be the Thom space of $\gamma_n(U,V)$.
There is a composition map
\[
\begin{array}{rcl}
\gamma_n(U,V) \times \gamma_n(V,W)
& \to &
\gamma_n(U,W) \\
( (f,x) , (g,y) )
& \mapsto &
(g \circ f , g(x) + y )
\end{array}
\]
which induces a composition on the corresponding Thom spaces.
Thus $\jcal_n$ has the structure of a category enriched over based topological spaces.
Furthermore, the standard action of the group $O(n)$ on $\rr^n$ induces
an action on the vector bundles
$\gamma_n(U,V)$ that is compatible with the composition maps above.
Hence $\jcal_n$ is a category enriched
over based topological spaces with an action of $O(n)$.

Since we have encountered spaces with $O(n)$-action, we briefly recap
that notion. 
We say that a map $f \co X \to Y$ of $O(n)$-spaces is \textbf{$O(n)$-equivariant}
if $\sigma f \sigma^{-1} =f$ for all $\sigma \in O(n)$.
We let $O(n) \Top$ denote the category of
based topological spaces with an $O(n)$-action
and $O(n)$-equivariant maps.
Recall that the base point is required to be $O(n)$-fixed.
The category $O(n) \Top$ is closed symmetric monoidal.
The monoidal product is given by the smash product of spaces, with $O(n)$ acting diagonally.
The internal function object is given by the space of (non-equivariant) maps
with $\sigma \in O(n)$ acting by conjugation: $f \mapsto \sigma f \sigma^{-1}$.
See \cite[Section II.1]{mm02} for details.

Now we can consider functors from $\jcal_n$ to $O(n) \Top$ that
preserve both the topological structure and the $O(n)$-action.
\begin{definition}
For $n \geqslant 0$, define $O(n) \ltimes (\jcal_n \Top)$ to be the
category of functors enriched over $O(n) \Top$ from
$\jcal_n$ to $O(n) \Top$. Thus if $F \in O(n) \ltimes (\jcal_n \Top)$
we have a map
\[
F_{U,V} \co \jcal_n(U,V) \longrightarrow \Top(F(U),F(V))
\]
for any $U,V \in \jcal_n$ which is continuous and is equivariant
with respect to the $O(n)$ action on the domain and codomain.
\end{definition}

The category $O(n) \ltimes (\jcal_n \Top)$ was denoted $O(n) \ecal_n$ in \cite{barnesoman13}.
There is an adjunction between this category and spectra with an action of $O(n)$, 
see \cite[Section 8]{barnesoman13} for details.  

\begin{definition}
There is a morphism of enriched categories $\alpha_n \co \jcal_n \to \jcal_1$
which sends $V$ to $\rr^n \otimes V$ and acts on mapping spaces by
\[
\begin{array}{rcl}
\jcal_n(U,V) & \to & \jcal_1(\rr^n \otimes U, \rr^n \otimes V) \\
(f,x) & \mapsto & (\rr^n \otimes f, x)
\end{array}
\]
This induces a functor
$\alpha_n^* \co O(n) \ltimes (\jcal_n \Top) \to O(n)\Sp$.
It is defined as $(\alpha_n^* E)(V) = E(\rr^n \otimes V)$,
but the action on the $O(n)$-space $E(\rr^n \otimes V)$ is altered to
use both the pre-existing action and also the action induced by $O(n)$
acting on $\rr^n \otimes V$. This has a left adjoint called
$\jcal_n \smashprod_{\jcal_1} (-)$.
\end{definition}

We also want to compare $O(n) \ltimes (\jcal_n \Top) $ and 
$\jcal_0 \Top$. 
The inclusion $\{ 0 \}= \rr^0 \to \rr^n$ induces a map of vector bundles
$\gamma_0 \to \gamma_n$ and a map of enriched categories
$i_n \co \jcal_0 \to \jcal_n$. This creates an adjoint pair between 
$O(n) \ltimes (\jcal_n \Top)$ and $\jcal_0 \Top$. 

\begin{definition}
Let $E \in O(n) \ltimes (\jcal_n \Top)$, then we can define an object
$\res_0^n/O(n) E$ of $\jcal_0 \Top$ by $V \mapsto E(V)/O(n)$.
This defines a functor $\res_0^n/O(n)$ from $O(n) \ltimes (\jcal_n \Top)$
to $\jcal_0 \Top$. There is a right adjoint to this functor, written $\ind_0^n$:
\[
(\ind_0^n F)(V) = \nat_{\jcal_0 \Top} (\res_0^n \jcal_n(V,-), F).
\]
The $O(n)$-action on the space $(\ind_0^n F)(V)$
is induced by the $O(n)$ action on $\jcal_n(V,-)$.
We call $\ind_0^n F$ the \textbf{$n^{th}$-derivative of $F$}
\end{definition}

We thus have adjunctions as below, with left adjoints on top.
\[
\xymatrix@C+1cm{
O(n)\Sp
\ar@<-1ex>[r]_-{\alpha_n^*}
&
\ar@<-1ex>[l]_-{(-)\smashprod_{\jcal_n}\jcal_1}
O(n) \ltimes (\jcal_n \Top)
\ar@<+1ex>[r]^-{\res_0^n/O(n)}
&
\ar@<+1 ex>[l]^-{\ind_0^n \varepsilon^*}
\jcal_0 \Top
}
\]

\subsection{The stable model structures}

We need to equip the middle and left hand categories with
model structures so that these adjunctions become Quillen equivalences
when $\jcal_0 \Top$ has the $n$-homogeneous model structure.
Our starting place is to put a
model structure on $O(n) \Top$ as below.
One can check this model
structure exists by lifting the model structure
on based spaces over the functor $O(n)_+ \smashprod -$, see
\cite[Theorem 11.3.2]{hir03}.
\begin{lemma}
The category $O(n) \Top$ of based spaces with $O(n)$-action
has a cellular and proper model structure, with generating sets given by
\[
\begin{array}{rcl}
I_{O(n) \Top} & = & \{
O(n)_+ \smashprod S^{k-1}_+ \longrightarrow O(n)_+ \smashprod D^k_+ \ | \
k \geqslant 0 \} \\
J_{O(n) \Top} & = & \{
O(n) \smashprod D^{k}_+ \longrightarrow O(n)_+ \smashprod (D^k \times [0,1])_+ \ | \
k \geqslant 0  \}
\end{array}
\]
The cofibrant objects have no fixed points (except the base point),
the weak equivalences are the underlying weak homotopy equivalences.
\end{lemma}

Using this model structure on $O(n)$-spaces,
we can equip $O(n) \ltimes (\jcal_n \Top)$ and $O(n) \Sp$
with \textbf{levelwise model structures} similar
to that for $\jcal_0 \Top$. These model structures are
proper and cellular. We want to make them into stable model categories.
For that, we need a new class of weak equivalences.
The idea is to generalise the
notion of $\pi_*$-isomorphisms of spectra.
For full details, see \cite[Section 7]{barnesoman13}.

For $V$ a vector space, $\rr^n \otimes V$
has the $O(n)$-action induced by the standard representation of $O(n)$ on $\rr^n$.
Let $S^{nV}$ be the one-point compactification of $\rr^n \otimes V$, then for 
$E \in O(n) \ltimes (\jcal_n \Top)$ there are $O(n)$-equivariant maps
\[
E(U) \smashprod S^{nV} \to E(U \oplus V).
\]
Just as with the definition of the stable homotopy groups of a spectrum, 
we can use these maps (and $n$-fold suspension) to construct 
the $n \pi_*$-homotopy groups of $E$. 

\begin{proposition}
The categories $O(n) \Sp$ and $O(n) \ltimes (\jcal_n \Top)$
have \textbf{stable model structures} where the weak equivalences
are the $\pi_*$ and $n\pi_*$-isomorphisms respectively.
The cofibrant objects are objectwise $O(n)$-free. The fibrant objects are
those whose adjoints of the suspension maps are weak homotopy equivalences,
so for $E \in O(n) \ltimes (\jcal_n \Top)$, the map
\[
E(U) \longrightarrow \Omega^{nV} E(U \oplus V)
\]
is a weak homotopy equivalence.
These model structures are stable, cellular and proper.
\end{proposition}
\begin{proof}
These model structures can be constructed as left Bousfield localisations.
We localise the levelwise model structures at the following sets of maps,
the first is for $O(n) \ltimes (\jcal_n  \Top)$ and the second is for
$O(n) \Sp$.
\[ \begin{array}{c}
\{
\jcal_n(V \oplus \rr, -) \smashprod S^n
\longrightarrow
\jcal_n(V, -)
 \ | \ V \in \skel \lcal
\}\\
\{
\jcal_1(V \oplus \rr, -) \smashprod S^1
\longrightarrow
\jcal_1(V, -)
 \ | \ V \in \skel \lcal
\}
\end{array}
\]
See \cite[Section 7]{barnesoman13} for a fuller discussion.
\end{proof}

Note that when $n=1$, $O(n) \ltimes (\jcal_n \Top)$ and
$O(n) \Sp$ are NOT the same category.
Let $E \in O(1) \ltimes (\jcal_1 \Top)$ and
$F \in O(n) \Sp$. Then for $U$ and $V$
in $\jcal_1$ we have a map of $O(n)$-equivariant spaces
\[
E_{U,V} \co \jcal_1(U,V) \longrightarrow \Top(E(U), E(V))
\]
and a map of spaces (where the $O(1)$ action on
$\jcal_1(U,V)$ is forgotten)
\[
F_{U,V} \co \jcal_1(U,V) \longrightarrow \Top(F(U), F(V))^{O(1)}.
\]

We now need to rationalise  $O(n) \ltimes (\jcal_n  \Top)$ and
$O(n) \Sp$. Each is left proper and cellular,
so we may use \cite{hir03} to perform a left Bousfield localisation
of each of the categories.
Recall the set $J_{\qq \Top}$ of generating acyclic cofibrations
for the rational model structure on based spaces. We use
this set to rationalise spectra with an $O(n)$-action.

\begin{lemma}\label{lem:spectralocal}
We define $O(n) \Sp_\qq$ to be
the left Bousfield localisation of the
stable model structure on $O(n) \Sp$ at the set
\[
Q_n = \{
O(n)_+ \smashprod \jcal_1(U,-) \smashprod j \ | \ j \in J_{\qq \Top}
\}.
\]
The weak equivalences are those maps which induce
isomorphisms on rational stable homotopy groups
(or equally rational homology). The fibrant objects
are the levelwise $\h \qq$-local $\Omega$-spectra.
\end{lemma}
\begin{proof}
This follows from \cite[Lemma 8.6]{barnesroitzheimframings} and
\cite[Lemma 4.14]{barnesroitzheimstable}.
\end{proof}

Now we perform the same operation with
the stable model structure on $O(n) \ltimes (\jcal_n \Top)$.

\begin{definition}
Define the rational model structure on
$O(n) \ltimes (\jcal_n \Top)$ to be
the localisation of the stable model structure on
$O(n) \ltimes (\jcal_n \Top)$ with respect to
\[
Q_n'= \{
O(n)_+ \smashprod \jcal_n(U,-) \smashprod j \ | \ j \in J_{\qq \Top} \ U \in \skel \jcal_n
\}.
\]
We denote this model structure $O(n) \ltimes (\jcal_n \Top)_\qq$.
\end{definition}

It is easily checked that the fibrant objects
of $O(n) \ltimes (\jcal_n \Top)_\qq$ are the objectwise $\h \qq$-local
objects whose adjoints of suspensions maps are weak homotopy equivalences.
We are now ready to prove our main result, Theorem \ref{thm:nohomogmodel}.

\subsection{The Quillen equivalences}

\begin{proposition}\label{prop:QEstable1}
There is a Quillen equivalence
\[
\xymatrix@C+1cm{
O(n)\Sp_\qq
\ar@<-1ex>[r]_-{\alpha_n^*}
&
\ar@<-1ex>[l]_-{(-)\smashprod_{\jcal_n}\jcal_1}
O(n) \ltimes (\jcal_n \Top)_\qq.
}
\]
\end{proposition}
\begin{proof}
If we apply $\jcal_1 \smashprod_{\jcal_n}(-)$ to a map of the form
$O(n)_+ \smashprod \jcal_n(U,-) \smashprod j$ for $j$ some
trivial cofibration in $L_{\h \qq} \Top$
we obtain the map $O(n)_+ \smashprod \jcal_1(nU,-) \smashprod j$.
These maps are weak equivalences in $O(n) \Sp_\qq$ so we
have a Quillen pair between the localisations
by \cite[Proposition 3.3.18]{hir03}.
The adjunction $((-)\smashprod_{\jcal_n}\jcal_1, \alpha_n^*)$
on non-rationalised model categories is a Quillen equivalence
by \cite[Section 8]{barnesoman13}.
So we can apply \cite[Proposition 2.3]{hov01}
to see that the localised adjunction is a
Quillen equivalence.
\end{proof}

If we apply the left derived functor of
${\res_0^n/O(n)}$ to the maps in $Q_n'$
we obtain the set
\[
\jcal_n \smashprod J_{\qq \Top} =
\{
\jcal_n(U,-) \smashprod j \ | \ j \in J_{\qq \Top} \ U \in \skel \jcal_n
\}
\]
which we used to localise the $n$-homogeneous model structure.
Hence we obtain the following result by the same proof as
for Proposition \ref{prop:QEstable1}. Again we use \cite[Section 8]{barnesoman13}
to see that $({\res_0^n/O(n)},{\ind_0^n \varepsilon^*})$ is a Quillen equivalence
between the non-rationalised model categories.

\begin{proposition}\label{prop:indrational}
Define $n \homog \jcal_0 \Top_\qq$ to be the left Bousfield localisation
of the model category $n \homog \jcal_0 \Top$ at the set $\jcal_n \smashprod J_{\qq \Top}$.
There is a Quillen equivalence
\[
\xymatrix@C+1cm{
O(n) \ltimes(\jcal_n \Top)_{\qq}
\ar@<+1ex>[r]^-{\res_0^n/O(n)}
&
\ar@<+1 ex>[l]^-{\ind_0^n \varepsilon^*}
n \homog \jcal_0 \Top_\qq.
}
\]
\end{proposition}

Thus we have now shown that $n \homog \jcal_0 \Top_\qq$
has the correct homotopy category, namely the homotopy category of rational
spectra with an action of $O(n)$. We can identify the derived composite 
of the above adjunctions and learn more about 
the cofibrant-fibrant objects of $n \homog \jcal_0 \Top_\qq$.

\begin{corollary}\label{cor:nhomogisrational}
There is an equivalence of homotopy categories 
\[
\ho (O(n)\Sp_\qq)
\cong
\ho (n \homog \jcal_0 \Top_\qq).
\]
Let $\Theta$ be a rational spectrum with $O(n)$-action, 
then the image of $\Theta$ in the category $\ho (n \homog \jcal_0 \Top_\qq)$
is given by 
\[
V \mapsto \Omega^\infty ((\Theta \smashprod S^{\rr^n \otimes V})_{hO(n)}).
\]
Let $F$ be a cofibrant and fibrant object of $n \homog \jcal_0 \Top_\qq$.
Then $F$ is an objectwise $\h \qq$-local $n$-homogeneous functor.
\end{corollary}
\begin{proof}
Combining the two propositions above at the level of 
homotopy categories gives the equivalence. 
That the equivalence agrees 
with Weiss's classification theorem follows
from the proof of \cite[Theorem 10.1]{barnesoman13}.

Let $F$ be a cofibrant and fibrant object of $n \homog \jcal_0 \Top_\qq$.
Then $F$ is in particular $n$-homogeneous, so it defines
a spectrum $\Theta_F$ with $O(n)$-action.
We know that $\ind_0^n F$ is objectwise $\h \qq$-local,
hence the spectrum $\Theta_F$ is objectwise $\h \qq$-local
by the Quillen equivalences above.
Now consider the $n$-homogeneous functor defined by $\Theta_F$:
\[
V \mapsto \Omega^\infty((S^{nV} \smashprod \Theta_F)_{hO(n)}).
\]
This functor is objectwise $\h\qq$-local and it is objectwise
weakly equivalent to $F$, thus $F$ itself is objectwise $\h \qq$-local.
\end{proof}

Thus we have shown that $n \homog \jcal_0 \Top_\qq$
is a model for the homotopy theory of
$n$-homogeneous functors in $\jcal_0 \Top$ that
are objectwise $\h \qq$-local.
Secondly we have shown that such functors are determined, up to homotopy,
by rational spectra with an action of $O(n)$.

\addcontentsline{toc}{part}{Bibliography}
\bibliography{bibliography2}
\bibliographystyle{alpha}

\end{document}